\documentclass[a4paper]{amsart}

\usepackage{amssymb,amscd}
\usepackage{mathrsfs} 
\usepackage[all]{xy}

\usepackage[
backref,
pdfauthor={Han Wu},  
pdftitle={Non-invariance of weak approximation with Brauer-Manin obstruction for surfaces},
]{hyperref}
\usepackage[alphabetic,backrefs,lite]{amsrefs}  

\theoremstyle{plain}
\theoremstyle{definition}

\newtheorem{theorem}{Theorem}[section]
\newtheorem{thm}{Theorem}[section]
\newtheorem{lemma}[theorem]{Lemma}

\newtheorem{prop}[theorem]{Proposition}

\theoremstyle{definition}
\newtheorem{definition}[theorem]{Definition}

\newtheorem{qu}[theorem]{Question}

\newtheorem{conjecture}[theorem]{Conjecture}

\theoremstyle{remark}
\newtheorem{remark}[theorem]{Remark}

\renewcommand{\AA}{\mathbb{A}}

\newcommand{\QQ}{\mathbb{Q}}

\newcommand{\GG}{\mathbb{G}}
\newcommand{\PP}{\mathbb{P}}

\newcommand{\Lcal}{{\mathcal L}}
\newcommand{\Ocal}{{\mathcal O}}

\newcommand{\Lscr}{{\mathscr L}}

\DeclareMathOperator{\Gal}{Gal}

\DeclareMathOperator{\Spec}{Spec}
\DeclareMathOperator{\pr}{pr}

\usepackage[OT2,T1]{fontenc}
\DeclareSymbolFont{cyrletters}{OT2}{wncyr}{m}{n}
\DeclareMathSymbol{\Sha}{\mathalpha}{cyrletters}{"58}

\newcommand{\defi}[1]{\textsf{#1}} 

\newcommand{\BM}{Brauer-Manin }

\newcommand{\Br}{\textup{Br}}
\newcommand{\et}{\textup{\'et}}

\makeatletter
\g@addto@macro\bfseries{\boldmath}  
\makeatother

\begin{document}
	
	\begin{title}
		{Non-invariance of weak approximation with Brauer-Manin obstruction for surfaces}  
	\end{title}
	\author{Han Wu}
	\address{Hubei University,
		Faculty of Mathematics and Statistics,
		Hubei Key Laboratory of Applied Mathematics,
		No. 368, Friendship Avenue, Wuchang District, Wuhan, 
		Hubei, 430062, P.R.China.}
	\email{wuhan90@mail.ustc.edu.cn}
	\date{}




	\begin{abstract} 
		In this paper, we study the property of weak approximation 
		with \BM obstruction for surfaces with respect to field extensions of number fields. For any nontrivial extension of number fields $L/K,$ assuming a conjecture of M. Stoll, we construct a smooth, projective, and geometrically connected surface over $K$ such that it  satisfies weak approximation with \BM obstruction 
		off all archimedean places, while its base change to $L$ fails. Then we illustrate this construction with an explicit unconditional example.
	\end{abstract} 
	
	\maketitle

	\section{Introduction}
	
	\subsection{Background}
		For a proper scheme $X$ over a number field $K,$ if its $K$-rational points set $X(K)$ is nonempty, then its adelic points set $X(\AA_K)$ is nonempty. We assume $X(K)\neq\emptyset.$ Let $\Omega_K$ be the set of all nontrivial places of $K,$ and let $S\subset \Omega_K$ be a finite subset. Let $\pr^S\colon\AA_K\to \AA_K^S$ be the natural projection of ad\`eles to ad\`eles without $S$ components, which induces a natural projection $\pr^S\colon X(\AA_K)\to X(\AA_K^S).$ By the diagonal embedding, we always view $X(K)$ as a subset of $X(\AA_K)$ (respectively of $X(\AA_K^S)$). 
	We say that $X$
	satisfies \defi{weak approximation} (respectively \defi{weak approximation off $S$}) if $X(K)$ is dense in $X(\AA_K)$ (respectively in $X(\AA_K^S)$), cf. \cite[Chapter 5.1]{Sk01}. Manin \cite{Ma71} used the Brauer group of $X$ to define a closed subset $X(\AA_K)^{\Br}\subset X(\AA_K),$ and showed that this closed subset can explain some failures of nondensity of $X(K)$ in $X(\AA_K^S).$  The global reciprocity law gives an inclusion: $X(K)\subset X(\AA_K)^\Br.$ 
	We say that $X$ satisfies \defi{weak approximation with Brauer-Manin obstruction} (respectively \defi{with Brauer-Manin obstruction off $S$}) if $X(K)$ is dense in $X(\AA_K)^\Br$ (respectively in $pr^S(X(\AA_K)^\Br)$). 
	For a smooth, projective, and geometrically connected curve $C$ defined over a number field $K,$ we assume that the Tate-Shafarevich group and the rational points set of its Jacobian are both finite. By the dual sequence of Cassels-Tate, Skorobogatov \cite[Chapter 6.2]{Sk01} and Scharaschkin \cite{Sc99} independently observed that $C(K)=pr^{\infty_K}(C(\AA_K)^\Br).$ Stoll \cite{St07} generalized this observation, and made a conjecture that for any smooth, projective, and geometrically connected curve, it satisfies weak approximation with Brauer-Manin obstruction off $\infty_K\colon$ see Conjecture \ref{conjecture Stoll} for more details.
	
	\subsection{Question}
	Let $L/K$ be a nontrivial extension of number fields. Let $S\subset\Omega_K$ be a finite subset, and  let $S_L\subset \Omega_L$ be the subset of all places above $S.$ Given a smooth, projective, and geometrically connected variety $X$ over $K,$ let $X_L=X\times_{\Spec K} {\Spec L}$ be the base change of $X$ to $L.$ In this paper, we consider the following question.
	\begin{qu}\label{question on WA1}
		If the variety $X$ has a $K$-rational point, and satisfies weak approximation with Brauer-Manin obstruction off $S,$ must $X_L$ also satisfy weak approximation with Brauer-Manin obstruction off $S_L?$	
	\end{qu}

	\subsection{A negative answer to Question \ref{question on WA1}}
	For any number field $K,$ 	
	assuming Stoll's conjecture, Liang\cite{Li21} found a quadratic extension $L,$ and constructed	
	a $3$-fold to give a negative answer to Question \ref{question on WA1}.  When $L=\QQ(\sqrt{5})$ and $K=\QQ,$ using the construction method, he gave an unconditional example with explicit equations  in loc. cit.  The author \cite{Wu22b} generalized his argument to any nontrivial extension of number fields. The varieties constructed there, are $3$-folds. In this paper, we will prove the same statement for smooth, projective, and geometrically connected surfaces.

	For any nontrivial extension of number fields $L/K,$ assuming Stoll's conjecture, we have the following theorem to give a negative answer to Question \ref{question on WA1}.
	\begin{thm}[Theorem \ref{theorem main result: non-invariance of weak approximation with BMO}]
		For any nontrivial extension of number fields $L/K,$ assuming Stoll's conjecture, there exists a smooth, projective, and geometrically connected surface  $X$ defined over $K$ such that
		\begin{itemize}
			\item the surface $X$ has a $K$-rational point, and satisfies weak approximation with Brauer-Manin obstruction 
			off $\infty_K,$
			\item the surface $X_L$ does not satisfy weak approximation with Brauer-Manin obstruction 
			off $T$ for any finite subset $T\subset \Omega_L.$ 
		\end{itemize}
	\end{thm}

	When $K=\QQ$ and $L=\QQ(i),$ using the construction method given in Theorem \ref{theorem main result: non-invariance of weak approximation with BMO}, we give an explicit unconditional example  in Section \ref{subsection main example1}. The smooth, projective, and geometrically connected surface  $X$ is defined by the following equations: 
	\begin{equation*}
		\begin{cases}
			(w_0w_2+w_1^2+16w_2^2)(x_0^2+x_1^2-x_2^2)+(w_0w_1+w_1w_2)(x_0^2-x_1^2)=0\\ 
			w_1^2w_2=w_0^3-16w_2^3
		\end{cases}
	\end{equation*}
	in $\PP^2\times \PP^2$ with bi-homogeneous coordinates $(w_0:w_1:w_2)\times(x_0:x_1:x_2).$

	\subsubsection{Main ideas behind our construction in the proof of Theorem \ref{theorem main result: non-invariance of weak approximation with BMO}}
	Let $L/K$ be a nontrivial extension of number fields. 
	We find a smooth, projective, and geometrically connected curve $C$ such that $C(K)$ and $C(L)$ are both finite, nonempty, and that $C(K)\neq C(L).$ Then we construct  a pencil $\beta \colon X\to C$ of curves parametrized by $C$ such that the fiber of each point in $C(K)$ is isomorphic to one given curve denoted by $C_\infty,$ and that the fiber of each point in $C(L)\backslash C(K)$  is isomorphic to another given curve denoted by $C_0.$ By combining some fibration arguments with the functoriality of Brauer-Manin pairing, the arithmetic properties of $C_\infty$ and $C_0$ will determine those of $X.$ We carefully choose the curves $C_\infty$ and $C_0$ to meet the needs of the theorem.

	\section{Notation}
	Let $K$ be a number field, and let $\Ocal_K$ be the ring of its integers. Let $\Omega_K$ be the set of all nontrivial places of $K.$ Let $\infty_K\subset \Omega_K$ be the subset of all archimedean places, and let $\Omega_K^f=\Omega_K\backslash \infty_K.$ 
	For $v\in \Omega_K,$ let $K_v$ be the completion of $K$ at  $v.$ 
	Given a finite subset $S\subset \Omega_K,$ let $\AA_K$ (respectively $\AA_K^S$) be the ring of ad\`eles (ad\`eles without $S$ components) of $K.$  We fix an algebraic closure $\overline{K}$ of $K,$ and let $\Gamma_K=\Gal(\overline{K}/K).$ We always assume that a field $L$ is a finite extension of $K.$ Let $S_L\subset \Omega_L$ be the subset of all places above $S.$

	In this paper, a \defi{$K$-scheme} will mean a reduced, separated scheme of finite type over $K,$ and all geometric objects are $K$-schemes. A \defi{$K$-curve}  will mean a proper $K$-scheme such that every irreducible components are of dimension one. In particular, a $K$-curve may have more than one irreducible component, and may have singular points. We say that a $K$-scheme is a \defi{$K$-variety} if it is geometrically integral. Be cautious that in our definition, a integral $K$-scheme may not be a variety, i.e. it may have multiple geometrically irreducible components. Given a proper $K$-scheme $X,$ if $X(\AA_K)\neq \emptyset,$ let $pr^S\colon X(\AA_K)\to X(\AA_K^S)$ be the projection induced by the natural projection $pr^S\colon\AA_K\to \AA_K^S.$ All cohomology groups in this paper are Galois or \'etale cohomology groups, and let $\Br(X)=H^2_{\et}(X,\GG_{m}).$

	\section{Stoll's conjecture for curves}

	For a smooth, projective, and geometrically connected curve $C$ defined over a number field $K,$ if the Tate-Shafarevich group and the rational points set of its Jacobian are both finite, then by combining the Cassels-Tate pairing with the Brauer evaluation pairing, Skorobogatov \cite[Chapter 6.2]{Sk01} and Scharaschkin \cite{Sc99} independently observed that $C(K)=pr^{\infty_K}(C(\AA_K)^\Br).$ Stoll \cite[Theorem 8.6]{St07} generalized this observation. Furthermore, he \cite[Conjecture 9.1]{St07} made the following conjecture.

	\begin{conjecture}\cite[Conjecture 9.1]{St07}\label{conjecture Stoll}
		For any smooth, projective, and geometrically connected curve $C$ defined over a number field $K,$  the set $C(K)$ is dense in $pr^{\infty_K}(C(\AA_K)^\Br).$ In particular, if $C(K)$ is finite, then $C(K)=pr^{\infty_K}(C(\AA_K)^\Br).$
	\end{conjecture}

	\begin{remark}
		It is well known that for an elliptic curve over $\QQ$ of analytic rank $0,$ its Mordell-Weil group and Tate-Shafarevich group are both finite. By the dual sequence of Cassels-Tate, Conjecture \ref{conjecture Stoll} holds for this elliptic curve.
	\end{remark}

	The following definition and lemma have already been stated in the paper \cite{Wu22e}. We give them below for the convenience of reading.
	\begin{definition}(\cite[Definition 4.3]{Wu22e})\label{definition curve of type}
		Given a smooth, projective, and geometrically connected curve $C$ defined over a number field $K,$ let $L/K$ be a nontrivial extension of number fields. We say that a triple $(C,K,L)$  is of \defi{type $I$} if 
		\begin{itemize}
			\item the sets $C(K)$ and $C(L)$ are both finite and nonempty,
			\item $C(K)\neq C(L),$
			\item Stoll's conjecture \ref{conjecture Stoll} holds for the curve $C.$ 
		\end{itemize}
	\end{definition}

	\begin{lemma}(\cite[Lemma 4.4]{Wu22e})\label{lemma stoll conjecture}
		Let $L/K$ be a nontrivial extension of number fields. Suppose that Conjecture \ref{conjecture Stoll} holds for all smooth, projective, and geometrically connected curves defined over $K.$ Then there exists a smooth, projective, and geometrically connected curve $C$ defined over $K$ such that the triple $(C,K,L)$ is of type $I.$
	\end{lemma}

	The following lemma is a strong form of \cite[Lemma 6.1]{Wu22e}. It will be used to choose a dominant morphism from a given curve to $\PP^1.$

	\begin{lemma}\label{lemma choose base change morphism}
		Let $L/K$ be a nontrivial extension of number fields. Given a smooth, projective, and geometrically connected curve $C$ defined over $K,$ we assume that the triple $(C,K,L)$ is of type $I$ (Definition \ref{definition curve of type}). For any finite $K$-subscheme $R\subset \PP^1,$  there exists a dominant $K$-morphism $\gamma\colon  C\to \PP^1$ such that 
		\begin{itemize}{
				\item  $\gamma(C(K))=\{\infty\}\subset \PP^1(K),$
				\item $\gamma(C(L)\backslash C(K))=\{0\}\subset \PP^1(K),$
				\item $\gamma$ is \'etale over $R.$    }
		\end{itemize}
	\end{lemma}
	
	\begin{proof}
		The proof is along the same idea as the proof of \cite[Lemma 6.1]{Wu22e}, where the statement was shown for $R\subset \PP^1\backslash\{0,\infty\}.$
		We will put one more condition for choosing a rational function. Let $K(C)$ be the function field of $C.$ 
		Since $C(K)$ and $C(L)$ are both finite nonempty, and $C(K)\neq C(L),$ by the Riemann-Roch theorem, we can choose a rational function $\phi\in K(C)^\times\backslash K^\times$ such that
		\begin{itemize}
			\item the set of its poles contains $C(K),$
			\item the set of its zeros contains $C(L)\backslash C(K),$
			\item all poles and zeros are of multiplicity one.
		\end{itemize}
		Then this rational function $\phi$ gives a dominant $K$-morphism $\gamma_0\colon C\to \PP^1$ such that 
		\begin{itemize}
			\item $\gamma_0(C(L)\backslash C(K))=\{0\}\subset \PP^1(K),$
			\item $\gamma_0(C(K))=\{\infty\}\subset \PP^1(K),$
			\item $\gamma_0$ is \'etale over $\{0,\infty\}.$
		\end{itemize}
		Then the branch locus of $\gamma_0$ is finite and contained in $\PP^1\backslash\{0,\infty\}.$
		We can choose an automorphism $\varphi_{\lambda_0}\colon \PP^1\to \PP^1, (u:v)\mapsto (\lambda_0 u:v)$ with $\lambda_0\in K^\times$ such that the branch locus of $\gamma_0$ has no intersection with $\varphi_{\lambda_0}(R).$ Let $\gamma= (\varphi_{\lambda_0})^{-1}\circ\gamma_0.$ Then the morphism $\gamma$ is \'etale over $R,$ and satisfies other conditions. 
	\end{proof}

	\section{A negative answer to Question \ref{question on WA1}}
	For any number field $K,$ assuming Conjecture \ref{conjecture Stoll}, Liang \cite[Theorem 4.5]{Li21} found a quadratic extension $L,$ and constructed  a $3$-fold to give a negative answer to Question \ref{question on WA1}. The author \cite[Theorem 5.2.1]{Wu22b} generalized his result to any nontrivial extension of number fields. Although the strategies of these two papers are different, the methods used there are combining the arithmetic properties of Ch\^atelet surfaces with a construction method from Poonen \cite{Po10}. Thus 
	the varieties constructed there, are $3$-folds.  
	For any extension of number fields $L/K,$ assuming Conjecture \ref{conjecture Stoll}, in this section, we will construct a smooth, projective, and geometrically connected surface to give a negative answer to Question \ref{question on WA1}. The method that we will use, is to combine some fibration lemmas with the arithmetic properties of curves, whose irreducible components are projective lines. 
	
	\subsection{Preparation Lemmas} We state the following lemmas, which will be used for the proof of Theorem \ref{theorem main result: non-invariance of weak approximation with BMO}.

	The following fibration lemma has already been stated in the paper \cite{Wu22b}. We give them below for the convenience of the reader.
	\begin{lemma}(\cite[Lemma 5.1.1]{Wu22b})\label{lemma fiber criterion for wabm}
		Let $K$ be a number field, and let $S\subset \Omega_K$ be a finite subset.  Let $f\colon X\to Y$ be a $K$-morphism of proper $K$-varieties $X$ and $Y$. Suppose that
		\begin{enumerate}{
				\item\label{fiber criterion for wabm condition 1}  the set $Y(K)$ is finite,
				\item\label{fiber criterion for wabm condition 2}  the variety $Y$ satisfies weak approximation with Brauer-Manin obstruction off $S,$
				\item\label{fiber criterion for wabm condition 3}  for any $P\in Y(K),$ the fiber $X_P$ of $f$ over $P$ satisfies weak approximation off $S.$}
		\end{enumerate}
		Then the variety $X$ satisfies weak approximation with Brauer-Manin obstruction off $S.$
	\end{lemma}

	The following fibration lemma can be viewed as a modification of \cite[Lemma 5.1.2]{Wu22b} to fit into our context.
	
	\begin{lemma}\label{lemma fiber criterion for not wabm}
		Let $K$ be a number field, and let $S\subset \Omega_K$ be a finite subset.  Let $f\colon X\to Y$ be a $K$-morphism of proper $K$-varieties $X$ and $Y$. We assume that
		\begin{enumerate}{
				\item\label{fiber criterion for not wabm condition 1}  the set $Y(K)$ is finite,
				\item\label{fiber criterion for not wabm condition 2} there exists some $P\in Y(K)$ such that the fiber $X_P$ of $f$ over $P$ does not satisfy weak approximation with Brauer-Manin obstruction off $S.$   }
		\end{enumerate}
		Then the variety $X$ does not satisfy weak approximation with Brauer-Manin obstruction off $S.$
	\end{lemma}
	
	\begin{proof}
		By Assumption (\ref{fiber criterion for not wabm condition 2}), take a $P_0\in Y(K)$ such that the fiber $X_{P_0}$ does not satisfy weak approximation with Brauer-Manin obstruction off $S.$ Then there exist a finite nonempty subset $S'\subset\Omega_K\backslash S$ and a nonempty open subset 	
		$L=\prod_{v\in S'}U_v\times \prod_{v\notin S'}X_{P_0}(K_v)\subset X_{P_0}(\AA_K)$ such that $L\cap X_{P_0}(\AA_K)^{\Br}\neq \emptyset,$ but that $L\cap X_{P_0}(K)=\emptyset.$
		By Assumption (\ref{fiber criterion for not wabm condition 1}), the set $Y(K)$ is finite, so we can take a Zariski open subset $V_{P_0}\subset Y$ such that $V_{P_0}(K)=\{P_0\}.$ For any $v\in S',$ since $U_v$ is open in $X_{P_0}(K_v)\subset f^{-1}(V_{P_0})(K_v),$ we can take an open subset $W_v$ of $f^{-1}(V_{P_0})(K_v)$ such that $W_v\cap X_{P_0}(K_v)=U_v.$ 	
		Consider the open subset $N=\prod_{v\in S'}W_v\times \prod_{v\notin S'}X(K_v)\subset X(\AA_K),$ then $L\subset N.$ By the functoriality of Brauer-Manin pairing, we have $X_{P_0}(\AA_K)^{\Br}\subset  X(\AA_K)^{\Br}.$ So the set $N\cap X(\AA_K)^{\Br}\supset L\cap X_{P_0}(\AA_K)^{\Br},$ is nonempty. But $N\cap X(K)=N\cap X_{P_0}(K)=L\cap X_{P_0}(K)
		=\emptyset,$ which implies that $X$ does not satisfy weak approximation with Brauer-Manin obstruction off $S.$
	\end{proof}

	The following lemma states that a $K$-scheme with multiple geometrically irreducible components will violate weak approximation.
	\begin{lemma}\label{lemma prevariety not satifying WA}
		Let $K$ be a number field, and let $S\subset \Omega_K$ be a finite subset. Let $X$ be a $K$-scheme, which is not a $K$-variety, i.e. it has multiple geometrically irreducible components.  We assume $\prod_{v\in \Omega_K}X(K_v)\neq \emptyset,$ then the scheme
		$X$ does not satisfy weak approximation off $S.$
	\end{lemma}
	
	\begin{proof}
		Let $X^0$ be the smooth locus of $X.$ Claim that $X^0\subset X$ is an open dense subscheme.
		To see this, we note that, since $X$ is reduced and $K$ is of characteristic $0,$ the scheme $X$ is geometrically reduced. For any geometrically irreducible component of $X,$  by \cite[Chapter II. Corollary 8.16]{Ha97}, its smooth locus is open dense in this geometrically irreducible component. So the claim follows. From this claim, the schemes $X$ and $X^0$ have the same number of geometrically irreducible components.	
		
		By the assumption that $X$ has multiple geometrically irreducible components, let $X_1^0$ and $X_2^0$ be two different geometrically irreducible components of $X^0,$ defined over the number fields $K_1$ and $K_2$ respectively. By the Lang-Weil estimate \cite{LW54}, the varieties $X_1^0$ and $X_2^0$ have local points for almost all places of $K_1$ and $K_2$ respectively. By the \v{C}ebotarev density theorem, we can take two different places $v_1,v_2\in \Omega_K^f\backslash S$ such that $v_1,v_2$ split in $K_1$ and also in $K_2,$ and that $X_1^0(K_{v_1})\neq\emptyset$ and $X_2^0(K_{v_2})\neq\emptyset.$  Since $\prod_{v\in \Omega_K}X(K_v)\neq \emptyset,$ we consider a nonempty open subset $L=X_1^0(K_{v_1})\times X_2^0(K_{v_2})\times\prod_{v\in \Omega_K\backslash \{v_1,v_2\}}X(K_v)\subset \prod_{v\in \Omega_K}X(K_v).$ Since $X^0$ is smooth, and the varieties $X_1^0,~X_2^0$ are different geometrically irreducible components, we have $X_1^0(K_{v_1})\cap X_2^0(K_{v_1})=\emptyset,$ which implies $X(K)\cap L=\emptyset.$ Hence $X$ does not satisfy weak approximation off $S.$
	\end{proof}
	
	The following two lemmas state that two projective lines meeting at one point will violate weak approximation with Brauer-Manin obstruction.
	
	\begin{lemma}\label{lemma two projective lines Brauer group}
		Let $C$ be the curve defined over a number field $K$ by the homogeneous equation:
		$x_0^2-x_1^2=0$ in $\PP^2$ with homogeneous coordinates $(x_0:x_1:x_2).$ Then the natural restriction map $\Br(K)\to \Br(C),$ is an isomorphism. 
	\end{lemma}
	
	\begin{proof}
		Let $C_1$ and $C_2$ be two irreducible components of $C.$ Let $i_1,$ $i_2$ and $i_3$ be the natural embeddings of $C_1,$  $C_2$ and $C_1\cap C_2$ in $C$ respectively. Then we have the following sequence of \'etale sheaves on $C\colon$
		$$0\to \Ocal_C\to i_{1*}\Ocal_{C_1}\oplus i_{2*}\Ocal_{C_2}\to i_{3*}\Ocal_{C_1\cap C_2}\to 0,$$
		where the map $i_{2*}\Ocal_{C_2}\to i_{3*}\Ocal_{C_1\cap C_2}$ is the opposite of the restriction map, and the other maps are canonical restriction maps. By checking the exactness of this sequence at each geometric point of $C,$ and using \cite[Chapter II. Theorem 2.15]{Mi80}, it is exact.
		It gives rise to an exact sequence of \'etale sheaves on $C\colon$
		$$	0\to\GG_{m,C}\to i_{1*}\GG_{m,C_1}\oplus i_{2*}\GG_{m,C_2}\to i_{3*}\GG_{m,C_1\cap C_2}\to 0.$$
		Since the intersection $C_1\cap C_2$ is a rational point, this sequence splits. Using \'etale cohomology, for any integer $n\geq 0,$ we have an exact sequence:
		$$0\to H^n_{\et}(C,\GG_{m})\to H^n_{\et}(C,i_{1*}\GG_{m,C_1}\oplus i_{2*}\GG_{m,C_2})\to H^n_{\et}(C,i_{3*}\GG_{m,C_1\cap C_2})\to 0.$$
		Since $i_1,$ $i_2$ and $i_3$ are closed embeddings, by \cite[Chapter II. Corollary 3.6]{Mi80}, the functors $i_{1*},$ $i_{2*}$ and $i_{3*}$ are exact. Since $C_1$ and $C_2$ are isomorphic to $\PP^1,$ we have the following commutative diagram:
		$$ \xymatrix{
			0\ar[r]& H^n_{\et}(C,\GG_{m})\ar[r]\ar@{=}[d]& H^n_{\et}(C,i_{1*}\GG_{m,C_1}\oplus i_{2*}\GG_{m,C_2})\ar[r]\ar[d]^{\cong}& H^n_{\et}(C,i_{3*}\GG_{m,C_1\cap C_2})\ar[r]\ar[d]^{\cong}& 0\\
			0\ar[r]& H^n_{\et}(C,\GG_{m})\ar[r]& H^n_{\et}(\PP^1,\GG_{m})\oplus H^n_{\et}(\PP^1,\GG_{m})\ar[r]& H^n(\Gamma_K,\overline{K}^\times)\ar[r]& 0\\
		}$$
		with exact rows. By taking $n=2,$ we have an exact sequence: $$0\to \Br(C)\to \Br(K) \oplus \Br(K) \to \Br(K)\to 0.$$ So we have $\Br(K)\cong \Br(C).$	
	\end{proof}
	
	\begin{remark}
		In \cite{HS14}, Harpaz and Skorobogatov used another exact sequence of \'etale sheaves on $C$ (cf. Proposition 1.1 in loc. cit.) to calculate the Brauer group of $C.$ By an easy computation, this lemma can be gotten from their Corollary 1.5 in loc. cit.
	\end{remark}

	\begin{lemma}\label{lemma two projective lines not satifying WA}
		Let $K$ be a number field, and let $S\subset \Omega_K$ be a finite subset. Let $C$ be the curve defined over $K$ by the homogeneous equation:
		$x_0^2-x_1^2=0$ in $\PP^2$ with homogeneous coordinates $(x_0:x_1:x_2).$ Then 
		the curve $C$ does not satisfy weak approximation with Brauer-Manin obstruction off $S.$
	\end{lemma}
	
	\begin{proof}
		Since the curve $C$ has $K$-rational points and two irreducible components, by Lemma \ref{lemma prevariety not satifying WA}, it does not satisfy weak approximation off $S.$ By Lemma \ref{lemma two projective lines Brauer group}, we have $\Br(K)\cong \Br(C).$ So the curve $C$ does not satisfy weak approximation with Brauer-Manin obstruction off $S.$
	\end{proof}

	\begin{theorem}\label{theorem main result: non-invariance of weak approximation with BMO}
		For any nontrivial extension of number fields $L/K,$ assuming that Conjecture \ref{conjecture Stoll} holds over $K,$ there exists a smooth, projective, and geometrically connected surface  $X$ defined over $K$ such that
		\begin{itemize}
			\item the surface $X$ has a $K$-rational point, and satisfies weak approximation with Brauer-Manin obstruction 
			off $\infty_K,$
			\item the surface $X_L$ does not satisfy weak approximation with Brauer-Manin obstruction 
			off $T$ for any finite subset $T\subset \Omega_L.$ 
		\end{itemize}
	\end{theorem}
	
	\begin{proof}
		We will construct a smooth, projective, and geometrically connected surface $X.$ Let $C_\infty$ be the projective line defined over $K$ by the homogeneous equation:
		$x_0^2+x_1^2-x_2^2=0$ in $\PP^2$ with homogeneous coordinates $(x_0:x_1:x_2).$ Let $C_0$ be the curve defined over $K$ by the homogeneous equation:
		$x_0^2-x_1^2=0$ in $\PP^2$ with homogeneous coordinates $(x_0:x_1:x_2).$
		Let $(u_0:u_1)\times(x_0:x_1:x_2)$ be the coordinates of $\PP^1\times\PP^2,$ and let $s'=u_0(x_0^2+x_1^2-x_2^2)+u_1(x_0^2-x_1^2)\in \Gamma(\PP^1\times\PP^2,\Ocal(1,2)).$ Let $X'$ be the locus defined by $s'=0$ in $\PP^1\times\PP^2.$ Since the curves $C_\infty$ and $C_0$ meet transversally, the locus $X'$ is smooth. Let $R$ be the locus over which the composition $ X'\hookrightarrow \PP^1\times\PP^2  \stackrel{pr_1}\to\PP^1$ is not smooth.  Then by \cite[Chapter III. Corollary 10.7]{Ha97}, it is finite over $K.$ 
		By the assumption that Conjecture \ref{conjecture Stoll} holds over $K,$ and Lemma \ref{lemma stoll conjecture}, we can take a smooth, projective, and geometrically connected curve $C$ defined over $K$ such that the triple $(C,K,L)$ is of type $I.$ By Lemma \ref{lemma choose base change morphism}, we can choose a $K$-morphism $\gamma\colon  C\to \PP^1$ such that $\gamma(C(L)\backslash C(K))=\{0\}\subset \PP^1(K),$ $\gamma(C(K))=\{\infty\}\subset \PP^1(K),$ and that $\gamma$ is \'etale over $R.$ 
		Let $B=C\times \PP^2,$ and let $(\gamma,id)\colon B\to\PP^1\times \PP^2.$  Let $\Lcal=(\gamma,id)^*\Ocal(1,2),$ and let $s=(\gamma,id)^* (s')\in \Gamma(B,\Lcal).$ Let $X$ be the zero locus of $s$ in $B.$
		Since $\gamma$ is \'etale over the locus  $R,$ the surface $X$ is smooth.
		Since $X$ is defined by the support of the global section $s,$ it is an effective divisor. The invertible sheaf $\Lscr (X')$ on $\PP^1\times\PP^2$ is isomorphic to $\Ocal(1,2),$ which is a very ample sheaf on $\PP^1\times\PP^2.$ And $(\gamma,id)$ is a finite morphism, so the pull back of this ample sheaf is again ample, which implies that the invertible sheaf $\Lscr (X)$ on $C\times\PP^2$ is ample. By \cite[Chapter III. Corollary 7.9]{Ha97}, the surface $X$ is geometrically connected. So the surface $X$ is smooth, projective, and geometrically connected. 
		Let $\beta\colon X \hookrightarrow B=C\times \PP^2 \stackrel{pr_1}\to C$ be the composition morphism. 
		By our construction, we have the following commutative diagram:
			$$\xymatrix{
				X \ar@{^(->}[d]\ar[r] \ar@/_2.5pc/[dd]_{\beta} & X'  \ar@{^(->}[d]  \\
				C\times \PP^2 \ar^{pr_1}[d] \ar^{(\gamma,id)}[r]&  \PP^1\times \PP^2 \ar^{pr_1}[d]\\
				C \ar^{\gamma}[r] & \PP^1  
			}$$

		Next, we will check that the surface $X$ has the properties.
		
		We will show that  $X$ has a $K$-rational point. For any $P\in C(K),$ we have  $\beta^{-1}(P)\cong C_\infty.$ The projective line $C_\infty$ has a $K$-rational point, so $X(K)\neq \emptyset.$\\
		We will show that  $X$ satisfies weak approximation with Brauer-Manin obstruction 
		off $\infty_K.$ 
		We consider the morphism $\beta.$ Since the projective line $C_\infty$ satisfies weak approximation, also weak approximation off $\infty_K,$	Assumption (\ref{fiber criterion for wabm condition 3}) of Lemma \ref{lemma fiber criterion for wabm} holds.
		Since Conjecture \ref{conjecture Stoll} holds for the curve $C,$ using Lemma \ref{lemma fiber criterion for wabm} for the morphism $\beta,$ the surface $X$ satisfies weak approximation with Brauer-Manin obstruction off $\infty_K.$ 
		
		For any finite subset $T\subset \Omega_L,$ we will show that  $X_L$ does not satisfy weak approximation with Brauer-Manin obstruction 
		off $T.$ We take a point  $Q\in C(L)\backslash C(K).$ By the choice of the curve $C$ and morphism $\beta,$ the fiber $\beta^{-1}(Q)$ is isomorphic to $C_{0L}.$
		By Lemma \ref{lemma two projective lines not satifying WA}, the curve $C_{0L}$ does not satisfy weak approximation with Brauer-Manin obstruction off $T\cup \infty_L.$ By Lemma \ref{lemma fiber criterion for not wabm}, the surface $X_L$ does not satisfy weak approximation with Brauer-Manin obstruction 
		off $T\cup \infty_L.$ So it does not satisfy weak approximation with Brauer-Manin obstruction 
		off $T.$
	\end{proof}

	\section{An Explicit unconditional example}\label{subsection main example1}

	In this section, let $K=\QQ$ and $L=\QQ(i).$ For this extension $L/K,$ we will give an explicit example without assuming Conjecture \ref{conjecture Stoll} for Theorem \ref{theorem main result: non-invariance of weak approximation with BMO}.

	\subsection{Choosing an elliptic curve} For the extension $L/K,$ we will choose an elliptic curve such that the triple $(E,K,L)$ is of type $I.$ Let $E$ be the elliptic curve defined over $\QQ$ by the homogeneous equation:
	$$w_1^2w_2=w_0^3-16w_2^3$$ in $\PP^2$ with homogeneous coordinates $(w_0:w_1:w_2).$ Its quadratic twist $E^{(-1)}$ is isomorphic to an elliptic curve defined by a homogeneous equation:
	$w_1^2w_2=w_0^3+16w_2^3.$ The elliptic curves $E$ and $E^{(-1)}$ over $\QQ,$ are  of analytic rank $0.$ Then the Tate-Shafarevich group $\Sha(E,K)$ is finite, so the curve $E$ satisfies weak approximation with Brauer-Manin obstruction off $\infty_K.$ The Mordell-Weil groups $E(K)$ and $E^{(-1)}(K)$ are both finite, so the group $E(L)$ is finite. Using \cite[SageMath]{St12}, we check that $E(K)=\{(0:1:0)\}$ and $E(L)=\{(0:\pm 4i:1),(0:1:0)\}.$ So the triple $(E,K,L)$ is of type $I.$

	\subsection{Choosing a dominant morphism}
	We choose the following dominant morphism from the elliptic curve $E$ to $\PP^1,$ which satisfies some conditions of Lemma \ref{lemma choose base change morphism}.

	Let $\PP^2\backslash\{(1:0:0),(-16:0:1),(-1:\pm \sqrt{15}i:1)\}\to \PP^1$ be a morphism over $\QQ$ given by $(w_0:w_1:w_2)\mapsto(w_0w_2+w_1^2+16w_2^2:w_0w_1+w_1w_2).$ Composing  the natural inclusion $E\hookrightarrow \PP^2\backslash\{(1:0:0),(-16:0:1),(-1:\pm \sqrt{15}i:1)\}$ with it, we get a morphism $\gamma\colon E\to \PP^1,$ which is a dominant morphism of degree $6.$ The dominant morphism $\gamma$ maps $E(K)$ to $\{\infty\}=\{(1:0)\},$ and maps $(0:\pm 4i:1)$ to $0:=(0:1).$  By B\'ezout's Theorem \cite[Chapter I. Corollary 7.8]{Ha97} and calculation, the branch locus of $\gamma$  is contained in $\PP^1\backslash\{\infty\}.$ Let $(u_0:1)\in \PP^1$ be a branch point of $\gamma.$ For fixed $u_0,$ we use the Jacobian criterion for the intersection of two curves $E$ and $w_0w_2+w_1^2+16w_2^2=(w_0w_1+w_1w_2)u_0$ in $\PP^2.$ Since the point $(0:1:0)\in \PP^2$ is not in this intersection, we 
	let $w_2=1$ to dehomogenize these two curves. By the Jacobian criterion, the branch locus satisfies the following equations:
	\begin{equation*}
		\begin{cases}
			w_1^2=w_0^3-16\\
			w_1^2+w_0+16=w_1(w_0+1)u_0\\
			3(2w_1-w_0u_0-u_0)w_0^2+2w_1(1-w_1u_0)=0.
		\end{cases}
	\end{equation*}
	Then 
	the branch locus equals 
	$$\left\{(u_0:1)\big{|} u_0^{12} + \frac{60627u_0^{10}}{4913} + \frac{159828u_0^8}{4913} - \frac{3505917u_0^6}{19652} - \frac{42057961u_0^4}{58956} + \frac{76076u_0^2}{14739} - \frac{4112}{132651}=0\right\}.$$ Let $(u_0:1)$ be a branch point, then the degree $[\QQ(u_0):\QQ]$ equals $12.$

	\subsection{Construction of a smooth, projective, and geometrically connected surface}
	We will construct a smooth, projective, and geometrically connected surface $X$ as in Theorem \ref{theorem main result: non-invariance of weak approximation with BMO}. 
	Let $(u_0:u_1)\times(x_0:x_1:x_2)$ be the coordinates of $\PP^1\times\PP^2,$ and let $s'=u_0(x_0^2+x_1^2-x_2^2)+u_1(x_0^2-x_1^2)\in \Gamma(\PP^1\times\PP^2,\Ocal(1,2)).$ The locus $X'$ defined by $s'=0$ in $\PP^1\times\PP^2$ is smooth. 
	Let $R$ be the locus over which the composition $ X'\hookrightarrow \PP^1\times\PP^2  \stackrel{pr_1}\to\PP^1$ is not smooth.  By calculation, 
	$R=\{(0:1),(\pm 1:1)\}.$ 
	Let $B=E\times \PP^2,$ and let $(\gamma,id)\colon B\to\PP^1\times \PP^2.$ Let $\Lcal=(\gamma,id)^*\Ocal(1,2),$ and let $s=(\gamma,id)^* (s')\in \Gamma(B,\Lcal).$ Let $X$ be the zero locus of $s$ in $B.$
	Since the locus $R$
	does not intersect with the branch locus of  $\gamma\colon E\to \PP^1,$ the surface $X$ is smooth.
	So it is smooth, projective, and geometrically connected. 
	By our construction, the surface $X$ is defined by the following equations: 
	\begin{equation*}
		\begin{cases}
			(w_0w_2+w_1^2+16w_2^2)(x_0^2+x_1^2-x_2^2)+(w_0w_1+w_1w_2)(x_0^2-x_1^2)=0\\ 
			w_1^2w_2=w_0^3-16w_2^3
		\end{cases}
	\end{equation*}
	in $\PP^2\times \PP^2$ with bi-homogeneous coordinates $(w_0:w_1:w_2)\times(x_0:x_1:x_2).$
	For this surface $X,$ we have the following proposition.

	\begin{prop}\label{example1: main result: non-invariance of weak approximation with BMO}
		For $K=\QQ$ and $L=\QQ(i),$  the smooth, projective, and geometrically connected surface $X$ has the following properties.	
		\begin{itemize}
			\item The surface $X$ has a $K$-rational point, and satisfies weak approximation with Brauer-Manin obstruction 
			off $\infty_K.$
			\item The surface $X_L$ does not satisfy weak approximation with Brauer-Manin obstruction 
			off $T$ for any finite subset $T\subset \Omega_L.$ 
		\end{itemize}
	\end{prop}

	\begin{proof}
		This is the same as in the proof of Theorem \ref{theorem main result: non-invariance of weak approximation with BMO}.
	\end{proof}

	\begin{footnotesize}
		\noindent\textbf{Acknowledgements.} The author would like to thank my thesis advisor Y. Liang for proposing the related problems, papers and many fruitful discussions. Thank M. Borovoi for his interested in my work and many helpful suggestions, and A. P\'al for kindly pointing out some inaccuracies. This paper was inspired by the work of Harpaz and Skorobogatov \cite{HS14}. The author was partially supported by the National Natural Science Foundation of
		China (NSFC) Grant No. 12071448 and funded by Open Foundation of Hubei Key Laboratory of Applied Mathematics (Hubei University).
	\end{footnotesize}


\end{document}